\documentclass[11pt]{amsart}
\usepackage[margin=1.2in]{geometry}
\usepackage{subfig}
\usepackage{amsmath,amssymb,amsfonts} 
\usepackage{graphicx} 
\usepackage[usenames]{color}
\usepackage{enumitem}
\newtheorem{theorem}{Theorem}[section]
\newtheorem{lemma}[theorem]{Lemma}

\theoremstyle{definition}

\newtheorem{example}[theorem]{Example}

\theoremstyle{remark}

\newcommand{\ep}{\epsilon}

\numberwithin{equation}{section}

\usepackage{graphicx}
\usepackage{amsmath,amsfonts,amssymb,amsthm}
\usepackage{mathtools}
\usepackage{commath}

\begin{document}

\title{Inexact Online Proximal Mirror Descent for time-varying composite optimization}

\author{Woocheol Choi}
\address{Department of Mathematics, Sungkyunkwan University, Suwon 440-746, Korea (Republic of)}
\email{choiwc@skku.edu}
\author{Myeong-Su Lee }
\address{Department of Mathematics, Sungkyunkwan University, Suwon 440-746, Korea (Republic of)}
\email{msl3573@skku.edu}
\author{SEOK-BAE YUN }
\address{Department of Mathematics, Sungkyunkwan University, Suwon 440-746, Korea (Republic of)}
\email{sbyun01@skku.edu}

\maketitle

\begin{abstract}
	In this paper, we consider the online proximal mirror descent for solving the time-varying composite optimization problems. For various applications, the algorithm naturally involves the errors in the gradient and proximal operator. We obtain sharp estimates on the dynamic regret of the algorithm when the regular part of the cost is convex and smooth. If the Bregman distance is given by the Euclidean distance, our result also improves the previous work in two ways: (i) We establish a sharper regret bound compared to the previous work in the sense that our estimate does not involve $O(T)$ term appearing in that work. (ii) We also obtain the result when the domain is the whole space $\mathbb{R}^n$, whereas the previous work was obtained only for bounded domains. We also provide numerical tests for problems involving the errors in the gradient and proximal operator.
\end{abstract}
\section{Introduction}

Time-varying optimization has been gaining increasing attention in recent years, arising in various application fields such as target tracking \cite{DSH}, model predictive control \cite{PMR}, and machine learning \cite{Ziegel}. This type of optimization problem is characterized by cost functions, constraints, and feasible domains that vary over time \cite{DSBM,SDPLG}. In particular, the time-varying loss functions are often given by non-smooth functions. A common example of those is that the loss function is given by the sum of a non-smooth regularization part and a regular objective function. This regularization part effectively handles noise and sparsity, and also prevent over-fitting \cite{LL,lasso,ZHHT}. 

Let us consider such time-varying non-smooth composite optimization problems:
\begin{align*}
	\min_{x\in\Omega\subset\mathbb{R}^n}f_k(x):=g_k(x)+h_k(x),\ k=1,2,\cdots,
\end{align*}
where $\Omega$ is a convex domain in $\mathbb{R}^n$, and the function $g_k:\mathbb{R}^n\rightarrow\mathbb{R}$ is an objective function, and the function $h_k:\mathbb{R}^n\rightarrow\mathbb{R}$ is a non-smooth regularizer. The most popular optimization scheme for such problems is the online proximal gradient method \cite{BD,DSR}. In this paper, we investigate the performance of the online proximal mirror descent algorithm. The algorithm is a generalization of the online proximal gradient descent, building on the Mirror Descent(MD). MD was originally introduced by Nemirovsky and Yudin \cite{MY}, generalizing the standard gradient descent (GD). It is known that MD converges faster than GD provided the Bregman divergence of MD is chosen suitably \cite{BT, BMN}. MD has been used effectively for large-scale optimization problems \cite{DAJ, KBB, NL,NB, R, RB}. Further, we consider a worse situation where the player attains the full functional form of $h_k$ but only an inexact gradient $\nabla g_k(x_{k-1})+e_k$ with some error $e_k\in\mathbb{R}^n$ \cite{DSR,SRB} and the proximal part is computed inexactly \cite{ASD}. Under this environment, the inexact online proximal mirror descent is presented as follows:
\begin{align}\label{algorithm}
	x_k\approx_{\epsilon_k}\underset{x\in\Omega}{\arg\min}\Phi_{\lambda h_k}(x):=h_k(x)+\langle\nabla g_k(x_{k-1})+e_k,x\rangle+\frac{1}{\lambda}V_\omega(x,x_{k-1}),
\end{align}
where the Bregman divergence $V_\omega$ is given in Section 2, and the notation 
$$x_k\approx_{\epsilon_k}\underset{x\in\Omega}{\arg\min}\Phi_{\lambda h_k}(x)$$  means that for each $k$, there is a positive constant  $\epsilon_k$ such that
\begin{align}\label{inexact}
	\|x_k-\underset{x\in\Omega}{\arg\min}\Phi_{\lambda h_k}(x)\big\|\leq \epsilon_k.
\end{align}

As a performance measure, we consider the following dynamic regret $R_T$:
\begin{align*}
	R_T:=\sum_{k=1}^Tf_k(x_k)-\sum_{k=1}^Tf_k(x_k^*),
\end{align*}
where the point $x_k^*$ is an optimal point of the loss function $f_k$ at the time instant $k$. This measures the difference between what is done by the player and that by the optimal.
In the literature, a bound of the Dynamic regret was first addressed in \cite{HW} for non-differentiable, lipschitz and continuous objective function where the following dynamic regret bound
\begin{align*}
	R_T=O\left(\sqrt{T}\left(1+\sum_{k=1}^{T-1}\norm{\theta_{k+1}-\Phi_k(\theta_k)}\right)\right)
\end{align*}
was derived for the dynamic mirror descent (DMD) with a dynamical model $\Phi_k$.
The authors in \cite{DSR} considered the case where the objective functions are smooth and strongly convex but the regularizer are non-differentiable, which obtained the following regret bound of the online proximal gradient descent with inexact  gradient:
\begin{align*}
	R_T = O(1+\Sigma_T + E_T),
\end{align*}
where $\Sigma_T = \sum_{k=1}^{T} \|x_k^* -x_{k-1}^*\|$ and $E_T = \sum_{k=1}^T \|e_k\|$.
The work \cite{ASD} considered the online proximal gradient descent where not only the error in the gradient is considered but also the proximal part is solved approximately. In \cite{ASD}, the following dynamic regret bounds were obtained for the objective functions being smooth and strongly convex:
	\begin{equation*}
		R_T = O(1 + \Sigma_{T} + P_T + E_T),
	\end{equation*}
and for the objective functions being smooth and convex:
	\begin{equation}\label{eq-1-10}
		R_T = O(1+T + \Sigma_T +\bar{\Sigma}_T + P_T + \bar{P}_T + E_T),
	\end{equation}
	where $\bar{\Sigma}_T = \sum_{k=1}^T \|x_k^* - x_{k-1}^*\|^2$. Also, $P_T = \sum_{k=1}^T \ep_k$ and $\bar{P}_T = \sum_{k=1}^T \ep_k^2$.

In addition, we refer to the recent paper \cite{KMD} where the dynamic regret was obtained under the Polyak-Lojasiewicz condition. The asymptotical tracking error was studied in \cite{ASD,ZDH}. We also refer to the references \cite{DSST,HW,YHHX} for the static regret of the time-varying composite optimization.
 
The analysis of the online proximal gradient descent becomes more difficult when inexact gradient is used and the proximal part is solved inexactly. Moreover, a dynamic regret bound is more delicate to obtain when the strongly convexity assumption on the loss function is missing. In particular, we observe that the regret bound \eqref{eq-1-10} involves $O(T)$ which may not be small enough even in the sense of averaging $R_T/T$. The boundedness assumption on domain was also necessarily imposed in the above results when the loss function is not strongly convex. Upon these difficulties, we aim to establish sharp estimates on the dynamic regret of Inexact Online Proximal Gradient Descent for the convex case. Specifically, our main contributions are as follows. We obtain a bound of the dynamic regret not involving an $O(T)$ term. This definitely improves the previous bound \eqref{eq-1-10} obtained in \cite{DSR}, since our result is obtained for online proximal mirror descent which generalize the online proximal gradient descent. Furthermore, without the assumption of the strongly convexity, we obtain a dynamic regret bound under the unbounded domain condition, whereas the previous results \cite{ASD,HW,YHHX} requires bounded domains. The above mentioned existing works and our works are summarized in Table 1.
\begin{table}[t!]
	\centering
	\scriptsize
	\renewcommand{\arraystretch}{1.2}
	\caption{\bf Summary of the derivations of optimization algorithms}
	\begin{tabular} { | p{0.8cm}|p{1.5cm} | c | c | p{1.5cm}|c|}
		\hline
		Ref. & Objective function & Bound of Dynamic Regret  & Mirror & Inexact & domain \\ 
		\hline 
		&&&&&\\[-1em]
		
		&&&&&\\[-1em]
		\cite{HW} &  Convex \&  Lipschitz &
		$O\left(\sqrt{T}\left(1+\sum_{k=1}^{T-1}\norm{\theta_{k+1}-\Phi_k(\theta_k)}\right)\right)$
		& Yes &No & bounded\\ 
		&&&&&\\[-1em]
		\hline
		&&&&&\\[-1em]
		\cite{DSR}&S.C \& Smooth & \vtop{\hbox{}\hbox{}\hbox{$O(1+\Sigma_T + E_T)$}} & No & Inexact gradient& unbounded\\ 
		&&&&&\\[-1em]
		\hline
		&&&&&\\[-1em]
		\cite{ASD} &S.C \& Smooth & $O(1 + \Sigma_{T} + P_T + E_T)$& No&gradient \& proximal &unbounded \\
		&&&&&\\[-1em] 
		\hline
		\cite{ASD} &Convex \& Smooth & $O(1+T + \Sigma_T +\bar{\Sigma}_T + P_T + \bar{P}_T + E_T)$ & No&gradient \& proximal & bounded \\
		&&&&&\\[-1em] 
		\hline     
		&&&&&\\[-1em]
		This paper & Convex \& Smooth & $O(1 + \bar{\Sigma}_T+ P_T + \bar{P}_T + \bar{E}_T)$  &Yes & gradient \& proximal
		& bounded\\ 
		\hline
		&&&&&\\[-1em]
		This paper & Convex \& Smooth & $O\Big (1 + P_T + \overline{\Sigma}_T + (E_T +P_T + \Sigma_T)^2\Big)$  &Yes & gradient \& proximal
		& unbounded\\ 
		\hline
	\end{tabular}
	\centering\newline\newline\newline
	{\fontsize{8pt}{0}
		Here, S.C means the strongly convexity, and we used the following notations: $\Sigma_{T}=\sum_{k=1}^T\|x_k^*-x_{k-1}^*\|$, $\bar{\Sigma}_{T}=\sum_{k=1}^T\|x_k^*-x_{k-1}^*\|^2$, $E_T=\sum_{k=1}^T\|e_k\|$, $\bar{E}_T=\sum_{k=1}^T\|e_k\|^2$, $P_T=\Sigma_{k=1}^T\epsilon_k$, and $\bar{P}_T=\Sigma_{k=1}^T\epsilon_k^2$.}
\end{table}

This paper is organized as follows. In the following section, we give the assumption used throughout this paper and state the main theorems of this paper. In Section 3, we obtain preliminary estimates for \eqref{algorithm} based on the convexity of the cost functions and the properties of the Bregman divergence. Section 4 is devoted to proving our main theorems. In Section 5, the numerical tests for the algorithm \eqref{algorithm} are given.  

\section{Main results}
In this section, we first give the assumptions on the loss functions and introduce the Bregman divergence used in the algorithm \eqref{algorithm}. We will then state the convergence results of this paper.

Throughout the paper, we will consider the loss functions and the regularizer satisfying the following assumptions.\\
{\bf Assumption 1}
\begin{itemize}
\item  $g_k$ is a closed, convex and proper function with a $L_k$-lipschitz continuous gradient at each time $k=1,2,\cdots$. We denote $L=\max_{k=1,...,T}{\{L_k\}}$ throughout the paper. 
\item $h_k$ is a $B_k$-lipschitz continuous  and convex regularizer for all $k=1,2,\cdots$. We also denote $B=\underset{k=1,...,T}{\max}{\{B_k\}}$. 
\item $\Omega$ is a convex set in $\mathbb{R}^n$.
\end{itemize}
The Bregman divergence $V_\omega:\mathbb{R}^n\times\mathbb{R}^n\rightarrow\mathbb{R}$ in the algorithm \eqref{algorithm}, associated with the distance-measuring function $\omega:\mathbb{R}^n\rightarrow\mathbb{R}$, is given by
\begin{equation}\label{eq-1-20}
	V_\omega(x,y):=\omega(x)-\omega(y)-\langle\nabla\omega(y),x-y\rangle,
\end{equation} 
where the distance-measuring function $\omega$ is assumed to satisfy the following conventional assumption (see the literature \cite{LY,YHHX}):\\
{\bf Assumption 2} 
\begin{itemize}
	\item $\omega$ is $\sigma_\omega$-strongly convex and has $G_\omega$-lipschitz gradients.
\end{itemize}

Now we state the detail of our main results. First, we state the result when the domain $\Omega$ is bounded. 
\begin{theorem}\label{thm-1-1} (Bounded domain)
	Assume that the domain $\Omega$ is bounded and the step-size satisfies   $\lambda\leq\frac{2\sigma_\omega}{L}$. Then, for the sequence $\{x_k\}_{k=1}^T$ generated by the algorithm \eqref{algorithm} with a initial point $x_0\in\Omega$, we have the following dynamic regret bound: \begin{align*}
		\begin{split}
			R_T&=\sum_{k=1}^{T}(f_k(x_k)-f_k(x_k^*))
			\\
			&= O(1+\Sigma_T+\bar{\Sigma}_T + P_T +E_T).
		\end{split}
	\end{align*}
\end{theorem}
\noindent It is worth mentioning that the result of Theorem \ref{thm-1-1} does not involve a $O(T)$ term. If the terms $\Sigma_T$, $\bar{\Sigma}_T$, $P_T$, and $E_T$ has $o(T)$ growth rates, then $R_T/T$ could become small enough when $T$ is sufficiently large. Therefore, this result provides a sharper bound on the dynamic regret compared to the previous one in \cite{ASD}. However, it is still limited to bounded domains. In the following theorem, we achieve an estimate on the dynamic regret when the domain $\Omega$ is the whole space $\mathbb{R}^n$.

\begin{theorem}\label{thm-1-2} (Whole domain) Assume that $\Omega = \mathbb{R}^n$ and the step-size satisfies $\lambda\leq\frac{2\sigma_\omega}{L}$. Then, for the sequence $\{x_k\}_{k=1}^T$ generated by the algorithm \eqref{algorithm} with a initial point $x_0\in\Omega$, we have the following dynamic regret bound:
	\begin{align*}
		\begin{split}
		R_T&=\sum_{k=1}^{T}(f_k(x_k)-f_k(x_k^*))
		\\
		&=O\Big (1 + P_T + \overline{\Sigma}_T + (E_T +P_T + \Sigma_T)^2\Big).
		\end{split}
	\end{align*}
\end{theorem}

\noindent The bound obtained in Theorem \ref{thm-1-2} also does not involve $O(T)$. Moreover, the boundedness assumption on domains is dropped. The proofs of the above theorems are given in the next sections.

\section{Technical lemmas}
In this section, we establish two lemmas for proving the main theorems. As well-known, the Bregman divergence $V_w$ defined in \eqref{eq-1-20} satisfies the following identity
\begin{equation}\label{eq-1-1}
	\nabla V_{\omega}(y,x) = \nabla V_{\omega}(y,z) + \nabla V_{\omega}(z,x)
\end{equation}
for all $x,y,z \in \mathbb{R}^d$. Here we use the notation $\nabla V_\omega$ to denote the derivative of $V_\omega$ with respect to the first variable. Also, we have the following identity, which is sometimes called the Pythagorean theorem in the literature:
\begin{equation}\label{eq-2-11}
	\begin{split}
		(z-y)\cdot\nabla V_\omega(y, x)&=(\nabla\omega(x)-\nabla\omega(y))\cdot(y-z)\\
		&=V_\omega(z,x)-V_\omega(z,y)-V_\omega(y,x)
	\end{split}
\end{equation}
for all $x,y,z \in \mathbb{R}^d$.

\begin{lemma} 
	For the sequence $\{x_k\}_{k=1}^T$ generated by the algorithm \eqref{algorithm} with an initial point $x_0\in\Omega$, we have the following estimates:
	\begin{enumerate}
		\item If $\Omega =\mathbb{R}^n$, then for any $z\in\mathbb{R}^n$,
		\begin{align}\label{c}
			\begin{split}
				g_k(x_k)+h_k(y_k)&\leq g_k(z)+h_k(z)+\frac{L_k}{2}\norm{x_k-x_{k-1}}^2+ e_k\cdot(z-x_k)\\
				&\quad+\nabla h_k(y_k)\cdot (y_k-x_k)-\frac{1}{\lambda}( x_k-z)\cdot V_\omega(y_k,x_k)\\
				&\quad-\frac{1}{\lambda}( x_k-z)\cdot V_\omega(x_k,x_{k-1}),
			\end{split}
		\end{align}
		where $y_k=\underset{x\in\Omega}{\arg\min}\,\Phi_{\lambda h_k}(x)$ and $\nabla h_k (y_k) \in \partial h_k (y_k)$.
		\item If $\Omega$ is bounded convex domain, then for any $z \in \Omega$, 
		\begin{align}\label{c-2}
			\begin{split}
				g_k(x_k)+h_k(y_k)&\leq g_k(z)+h_k(z)+\frac{L_k}{2}\norm{x_k-x_{k-1}}^2+e_k\cdot(z-x_k)\\
				&\quad+(\nabla h_k(y_k)+Q_k)\cdot (y_k-x_k)-\frac{1}{\lambda}(x_k-z)\cdot \nabla V_\omega(y_k,x_k)\\
				&\quad-\frac{1}{\lambda}(x_k-z)\cdot\nabla V_\omega(x_k,x_{k-1}),
			\end{split}
		\end{align}
		where $Q_k = \nabla h_k (y_k) + \nabla g_k (x_{k-1}) + e_k + \frac{1}{\lambda} \nabla V_w (y_k, x_{k-1})$.
	\end{enumerate}
\end{lemma}
\begin{proof}

	Since $g_k$ has a $L_k$-Lipschitz continuous gradient, we have
	\begin{align}\label{gklip}
		g_k(x_k)\leq g_k(x_{k-1})+\nabla g_k(x_{k-1})\cdot (x_k-x_{k-1})+\frac{L_k}{2}\norm{x_k-x_{k-1}}^2,
	\end{align}
	for any $z \in \Omega$.
	On the other hand, the convexity of $g_k$ implies that
	\begin{align}\label{gkconv}
		g_k(x_{k-1})\leq g_k(z)+\nabla g_k(x_{k-1})\cdot(x_{k-1}-z).
	\end{align}
	Summing up (\ref{gklip}) and (\ref{gkconv}), we get
	\begin{align}\label{gk}
		g_k(x_k)\leq g_k(z)+\nabla g_k(x_{k-1})\cdot( x_k-z)+\frac{L_k}{2}\norm{x_k-x_{k-1}}^2.
	\end{align}
	\textbf{(Case 1)}  $\Omega = \mathbb{R}^n$.
	
	Since $y_k=\arg\min_x\Phi_{\lambda h_k}(x)$, there exists $\nabla h_k (y_k) \in \partial h_k (y_k)$ such that
	\begin{align*}
		0= \nabla h_k(y_k)+\nabla g_k(x_{k-1})+e_k+\frac{1}{\lambda}\nabla V_\omega(y_k,x_{k-1}).
	\end{align*}
	Using this in \eqref{gklip}, we get
	\begin{equation}\label{eq-2-1}
		\begin{split}
			g_k(x_k)&\leq g_k(z)+\left(-\nabla h_k(y_k)-e_k-\frac{1}{\lambda}\nabla V_\omega(y_k,x_{k-1})\right)\cdot( x_k-z)+\frac{L_k}{2}\norm{x_k-x_{k-1}}^2\\
			&=g_k(z)+\frac{L_k}{2}\norm{x_k-x_{k-1}}^2-e_k\cdot (x_k-z)\\
			&\qquad -\nabla h_k(y_k)\cdot (x_k-z)-\frac{1}{\lambda} \nabla V_\omega(y_k,x_{k-1})\cdot( x_k-z).
		\end{split}
	\end{equation}
	We use the convexity of $h$ to find
	\begin{align}\label{find}
		\begin{split}
			-\nabla  h_k(y_k)\cdot (x_k-z) &= \nabla h_k(y_k)\cdot(z-y_k) +\nabla h_k(y_k)\cdot(y_k-x_k) \\
			&\leq h_k(z)-h_k(y_k)+\nabla h_k(y_k)\cdot(y_k-x_k),
		\end{split}
	\end{align}
	and apply \eqref{eq-1-1} to find
	\begin{align}\label{find2}
		-\frac{1}{\lambda}(x_k-z)\cdot\nabla V_\omega(y_k,x_{k-1})&=-\frac{1}{\lambda}(x_k-z)\cdot (\nabla V_\omega(y_k,x_k)+\nabla V_\omega(x_k,x_{k-1})).
	\end{align}
	Inserting \eqref{find} and \eqref{find2} into \eqref{eq-2-1} leads to the desired estimate \eqref{c}.
	
	\medskip
	
	\noindent \textbf{(Case 2)} $\Omega$ is a bounded convex domain.

	In this case, for any $x \in \Omega$ we have
	\begin{equation*}
		(z-y_k)\cdot \left(\nabla h_k (y_k) + \nabla g_k (x_{k-1}) + e_k + \frac{1}{\lambda} \nabla V_w (y_k, x_{k-1})\right) \geq 0.
	\end{equation*}
	We writie this as
	\begin{equation*}
		(z-x_k)\cdot\left(\nabla h_k (y_k) + \nabla g_k (x_{k-1}) + e_k + \frac{1}{\lambda} \nabla V_w (y_k, x_{k-1})\right) + (x_k -y_k)\cdot Q_k\geq 0,
	\end{equation*}
	where $Q_k = \nabla h_k (y_k) + \nabla g_k (x_{k-1}) + e_k + \frac{1}{\lambda} \nabla V_w (y_k, x_{k-1})$.
	Using this in \eqref{gk}, we get
	\begin{equation*} 
		\begin{split}
			g_k(x_k)&\leq g_k(z)+\left( -\nabla h_k(y_k)-e_k-\frac{1}{\lambda}\nabla V_\omega(y_k,x_{k-1})\right) \cdot(x_k-z)\\
			&\qquad+(x_k -y_k)\cdot Q_k+\frac{L_k}{2}\norm{x_k-x_{k-1}}^2.
		\end{split}
	\end{equation*}
	Then, applying the same estimates \eqref{find} and \eqref{find2}, the desired estimate follows.
\end{proof} 
\begin{lemma}\label{lemma} Assume that the step-size $\lambda$ satisfies $\lambda\leq \frac{2\sigma_\omega}{L}$. For the sequence $\{x_k\}_{k=1}^T$ generated by the algorithm \eqref{algorithm} with an initial point $x_0\in\Omega$, the following inequality holds:
	\begin{align}\label{eq7} 
		\begin{split}
			&\sum_{k=1}^{T}(f_k(x_k)-f_k(x_k^*))+\frac{1}{\lambda}V_\omega(x_T^*,x_T)
			\\
			&\leq\frac{1}{\lambda}V_\omega(x_0^*,x_0)+\frac{G_\omega s_1}{\lambda} \|x_0^*-x_0\|+D\sum_{k=1}^{T}\epsilon_k+\frac{1}{\lambda}\bigg( G_\omega-\frac{\sigma_\omega}{2}\bigg)\sum_{k=1}^{T}s_k^2 \\
			&\quad+\sum_{k=1}^{T}\bigg(e_k+\frac{G_\omega s_{k+1}}{\lambda}+\frac{G_\omega\epsilon_k}{\lambda}\bigg)\norm{x_k-x_k^*},
		\end{split}
	\end{align}
	where the sequence $\{s_k\}_{k=1}^{T+1}$ and the constant $D$ are defined as
	\begin{equation*}
		s_k=\left\{\begin{array}{ll} \|x_k^*-x_{k-1}^*\|& \quad \textrm{if}~  k=1,\dots,T,
			\\
			0 &\quad \textrm{if}~ k=T+1,
		\end{array}\right.
	\end{equation*}
	and 
	\begin{equation*}
		D=\left\{\begin{array}{ll} 2B& \quad \textrm{if}~ \Omega = \mathbb{R}^n
			\\
			\underset{k=1,...,T}{\max} \|2B_k+\nabla g_k (x_{k-1}) + e_k + \frac{1}{\lambda} \nabla V_w (y_k, x_{k-1})\|&\quad \textrm{if}~ \textrm{$\Omega$ is bounded domain}.
		\end{array}\right.
	\end{equation*}
	
\end{lemma}

\begin{proof}  We only prove the lemma for the case $\Omega = \mathbb{R}^n$ since the same proof directly applies to the bounded domain case if we use \eqref{c-2} instead of \eqref{c}.

	We put $z=x_k^*$ in (\ref{c}) to find
	\begin{align}\label{eq5}
		\begin{split}
			g_k(x_k)+h_k(y_k)&\leq g_k(x_k^*)+h_k(x_k^*)+\frac{L_k}{2}\norm{x_k-x_{k-1}}^2 + e_k\cdot (x_k^* -x_k)\\
			&\quad+\nabla h_k(y_k)\cdot(y_k-x_k)-\frac{1}{\lambda}(x_k-x_k^*)\cdot\nabla V_\omega(y_k,x_k) \\
			&\quad-\frac{1}{\lambda} (x_k-x_k^*)\cdot\nabla V_\omega(x_k,x_{k-1}).
		\end{split}
	\end{align}
	We employ the identity \eqref{eq-2-11} to find
	\begin{align*}
		\begin{split}
		&(x_k^*-x_k)\cdot\nabla V_\omega(x_k,x_{k-1})\\
		&  =V_\omega(x_k^*,x_{k-1})-V_\omega(x_k^*,x_k)-V_\omega(x_k,x_{k-1})\\
		&=(V_\omega(x_{k-1}^*,x_{k-1})-V_\omega(x_k^*,x_k))+(V_\omega(x_k^*,x_{k-1})-V_\omega(x_{k-1}^*,x_{k-1}))-V_\omega(x_k,x_{k-1}).
		\end{split}
	\end{align*}
	Putting this into \eqref{eq5}, we get
	\begin{align}\label{eq6}
		\begin{split}
			&g_k(x_k)+h_k(y_k)\\
			&\leq f_k(x_k^*) + (y_k-x_k)\cdot \nabla h_k(y_k)+\frac{1}{\lambda}(V_\omega(x_{k-1}^*,x_{k-1})-V_\omega(x_k^*,x_k))+e_k\cdot (x_k-x_k^*)\\
			&\quad -\frac{1}{\lambda}( x_k-x_k^*)\cdot\nabla V_\omega(y_k,x_k)+\frac{L_k}{2}\norm{x_k-x_{k-1}}^2-\frac{1}{\lambda}V_\omega(x_k,x_{k-1})\\&\quad+\underbracket{\frac{1}{\lambda}\bigg(V_\omega(x_k^*,x_{k-1})-V_\omega(x_{k-1}^*,x_{k-1})\bigg)}_{A}.
		\end{split}
	\end{align}
	Using the $\sigma_\omega$-strongly convexity of $\omega$ and $G_\omega$-lipschitz continuity of $\nabla\omega$, we estimate $A$ as follows:
	\begin{align*}
		V_\omega(x_k^*,x_{k-1})-V_\omega(x_{k-1}^*,x_{k-1})&=-V_\omega(x_{k-1}^*,x_k^*)+(\nabla w(x_k^*)-\nabla\omega(x_{k-1}))\cdot (x_k^*-x_{k-1}^*)\\&
		\leq -\frac{\sigma_\omega}{2}\norm{x_{k-1}^*-x_k^*}^2+(\nabla w(x_{k-1}^*)-\nabla\omega(x_{k-1}))\cdot( x_k^*-x_{k-1}^*)\\
		&\quad +(\nabla w(x_k^*)-\nabla\omega(x_{k-1}^*))\cdot(x_k^*-x_{k-1}^*)\\&
		\leq \bigg( G_\omega-\frac{\sigma_\omega}{2}\bigg)\norm{x_{k-1}^*-x_k^*}^2+G_\omega\norm{x_{k-1}^*-x_{k-1}}\norm{x_{k-1}^*-x_k^*}.
	\end{align*}
	Using \eqref{inexact} and Assumption 1, we also have
	\begin{equation*}
		(y_k -x_k)\cdot\nabla h_k (y_k)\leq B\epsilon_k.
	\end{equation*}
	Inserting these estimates and $V_w (x_{k}, x_{k-1}) \geq \sigma_w \|x_k -x_{k-1}\|^2$ in \eqref{eq6}, we deduce
	\begin{equation}\label{eq-2-3}
		\begin{split}
			&g_k(x_k)+h_k(y_k)\\&\leq f_k(x_k^*) + B\epsilon_k  +\frac{1}{\lambda}\bigg(V_\omega(x_{k-1}^*,x_{k-1})-V_\omega(x_k^*,x_k)\bigg)+\|e_k\|\norm{x_k-x_k^*}\\&
			\quad+\frac{G_\omega}{\lambda}\norm{x_k-x_k^*}\epsilon_k+\bigg(\frac{L_k}{2}-\frac{\sigma_\omega}{\lambda}\bigg)\norm{x_k-x_{k-1}}^2\\&
			\quad+ \frac{1}{\lambda}\bigg( G_\omega-\frac{\sigma_\omega}{2}\bigg)\norm{x_{k-1}^*-x_k^*}^2+\frac{1}{\lambda}G_\omega\norm{x_{k-1}^*-x_{k-1}}\norm{x_{k-1}^*-x_k^*}.
		\end{split}
	\end{equation}
	Now we use the Lipschitz continuity of $h_k$ to find
	\begin{align*}
		h_k(y_k)=h_k(x_k)+h_k(y_k)-h_k(x_k)\geq h_k(x_k)-B\norm{y_k-x_k}\geq h_k(x_k)-B\epsilon_k.
	\end{align*} 
	Putting this in \eqref{eq-2-3} and using that $\lambda\leq \frac{2\sigma_\omega}{L_k}$, we have
	\begin{align*}
		f_k(x_k)&=g_k(x_k)+h_k(x_k)\\
		&\leq f_k(x_k^*) +2 B\epsilon_k+\frac{1}{\lambda}\bigg(V_\omega(x_{k-1}^*,x_{k-1})-V_\omega(x_k^*,x_k)\bigg)+\|e_k\|\norm{x_k-x_k^*}\\
		&\quad+ \frac{1}{\lambda}\bigg( G_\omega-\frac{\sigma_\omega}{2}\bigg)\|x_{k-1}^*-x_{k}^*\|^2+\frac{G_\omega}{\lambda}\norm{x_k-x_k^*}\epsilon_k\\
		&\quad+\frac{G_\omega}{\lambda} \|x_{k-1}^*-x_{k}^*\|\norm{x_{k-1}^*-x_{k-1}}.
	\end{align*}
	Summing this from $k=1$ to $k=T$ leads to the desired estimate \eqref{eq7}.
\end{proof}

\section{Proofs of main results}\label{proof}
In this section, we prove Theorem \ref{thm-1-1} and Theorem \ref{thm-1-2}. Both proofs are based on the estimate \eqref{eq7} of Lemma \ref{lem-2-4}, but we need to carefully find a bound of the term $\|x_k-x_k^*\|$ for Theorem \ref{thm-1-2} since the domain $\Omega$ is given by the whole space $\mathbb{R}^n$. 

\subsection{Proof of Theorem \ref{thm-1-1}} We recall the inequality \eqref{eq7} of Lemma \ref{lemma}:
	\begin{align*}
		\begin{split}
			&\sum_{k=1}^{T}(f_k(x_k)-f_k(x_k^*))+\frac{1}{\lambda}V_\omega(x_T^*,x_T)
			\\
			&\leq\frac{1}{\lambda}V_\omega(x_0^*,x_0)+\frac{G_\omega s_1}{\lambda}\|x_0^*-x_0\|+D\sum_{k=1}^{T}\epsilon_k+\frac{1}{\lambda}\bigg( G_\omega-\frac{\sigma_\omega}{2}\bigg)\sum_{k=1}^{T}s_k^2 \\
			&\quad+\sum_{k=1}^{T}\bigg(e_k+\frac{G_\omega s_{k+1}}{\lambda}+\frac{G_\omega\epsilon_k}{\lambda}\bigg)\norm{x_k-x_k^*}.
		\end{split}
	\end{align*}
	Here, the term $\frac{1}{\lambda}V_\omega(x_T^*,x_T)$ in the left hand side is strictly positive due to the strongly convexity of $\omega$. Then, if we denote $R$ by the diameter of $\Omega$, then we have that $\|x_k-x_k^*\|\leq R$ for all $k\leq T$. Combining these facts, we derive the following estimate,
	\begin{align*}
		\begin{split}
			\sum_{k=1}^{T}(f_k(x_k)-f_k(x_k^*))
			&\leq\frac{1}{\lambda}V_\omega(x_0^*,x_0)+\frac{RG_\omega s_1}{\lambda}+D\sum_{k=1}^{T}\epsilon_k+\bigg( G_\omega-\frac{\sigma_\omega}{2}\bigg)\sum_{k=1}^{T}s_k^2 \\
			&\quad+R\sum_{k=1}^{T}\bigg(e_k+G_\omega s_{k+1}+\frac{G_\omega\epsilon_k}{\lambda}\bigg).
		\end{split}
	\end{align*}
	Rearranging the index of summation for $s_k$, we obtain the desired estimate. \qed \vspace{10pt}\\ 

 Now we turn to prove Theorem \ref{thm-1-2}, where the domain $\Omega$ is the whole space $\mathbb{R}^n$. For this case, it is a nontrivial issue to obtain a reasonable bound for $\|x_k-x_k^*\|$. To bound the term $\|x_k-x_k^*\|$ in a recursive way, we will make use of the following lemma:
 \begin{lemma}[\cite{SRB}]\label{lem-2-4} Assume that the non-negative sequence $\{u_k\}$ satisfies the following recursion for all $i\geq1$:\begin{align*}
 		u_i^2\leq S_i+\sum_{k=1}^{i}\tau_ku_k.
 	\end{align*}
 	with $\{S_k\}$ an increasing sequence, $S_0\geq u_0^2$ and $\tau_i\geq0$ for all $i$. Then, for all $i\geq1$, then\begin{align*}
 		u_i\leq\frac{1}{2}\sum_{k=1}^{i}\tau_k+\bigg(S_i+\bigg(\frac{1}{2}\sum_{k=1}^{i}\tau_k\bigg)^2\bigg)^{1/2}.
 	\end{align*}
 \end{lemma}
 
 \subsection{Proof of Theorem \ref{thm-1-2}} 
 By the estimate \eqref{eq7}, we have for $i\geq 1$ the following inequality
 \begin{align*}
 	\begin{split}
 		&\sum_{k=1}^{i}(f_k(x_k)-f_k(x_k^*))+\frac{1}{\lambda}V_\omega(x_i^*,x_i)\\
 		&\leq\frac{1}{\lambda}V_\omega(x_0^*,x_0)+\frac{G_\omega s_1}{\lambda}\|x_0^*-x_0\|+D\sum_{k=1}^{i}\epsilon_k+\frac{1}{\lambda}\bigg( G_\omega-\frac{\sigma_\omega}{2}\bigg)\sum_{k=1}^{i}s_k^2 \\
 		&\quad+\sum_{k=1}^{i}\bigg(e_k+\frac{G_\omega s_{k+1}}{\lambda}+\frac{G_\omega\epsilon_k}{\lambda}\bigg)\norm{x_k-x_k^*}.
 	\end{split}
 \end{align*} 
 This, together with the fact
 \begin{align*}
 	f_k(x_k)-f_k(x_k^*)\geq 0\quad \forall ~ k \geq 1,
 \end{align*}
 gives that for any positive integer $i$,
 \begin{align}\label{inequality1}
 	\begin{split}
 		\frac{\sigma_\omega}{2\lambda}\norm{x_i^*-x_i}^2&\leq\frac{1}{\lambda}V_\omega(x_i^*,x_i)\\&\leq Z_0+D\sum_{k=1}^{i}\epsilon_k+\frac{1}{\lambda}\bigg( G_\omega-\frac{\sigma_\omega}{2}\bigg)\sum_{k=1}^{i}s_k^2
 		\\&\quad+\sum_{k=1}^{i}\bigg(e_k+\frac{G_\omega s_{k+1}}{\lambda}+\frac{G_\omega\epsilon_k}{\lambda}\bigg)\norm{x_k-x_k^*},
 	\end{split}
 \end{align}
 where we used the strongly convexity of $\omega$ in the first line and the following notation
 \begin{equation*}
 	Z_0 = \frac{1}{\lambda}\Big(V_\omega(x_0^*,x_0)+ G_w s_1 \|x_0^* -x_0\|\Big).
 \end{equation*}
 
\vspace{10pt}
\noindent Now we set the following variables
\begin{align}\label{definitions}
	\begin{split}
	u_i&=\norm{x_i-x_i^*},\\
	S_i&=\frac{2\lambda}{\sigma_\omega}Z_{0} +\frac{2\lambda D}{\sigma_\omega}\sum_{k=1}^{i}\epsilon_k+\bigg( \frac{2 G_\omega}{\sigma_\omega}-1\bigg)\sum_{k=1}^{i}s_k^2,\\
	\tau_i &= \frac{2\lambda}{\sigma_\omega}e_i+\frac{2G_\omega}{\sigma_\omega} s_{i+1}+\frac{2G_\omega}{\sigma_\omega}\epsilon_i,
	\end{split}
\end{align}
and apply Lemma \ref{lem-2-4} on \eqref{inequality1} to obtain that for $i\leq T$,
\begin{equation*}
	\begin{split}
		\|x_i - x_i^* \| &\leq \frac{1}{2} \sum_{k=1}^i \tau_k + \Big( S_i + \Big( \frac{1}{2} \sum_{k=1}^i \tau_k\Big)^2 \Big)^{1/2} 
		\\
		& \leq \sum_{k=1}^i \tau_k + S_i^{1/2},\\
		&  \leq \sum_{k=1}^T \tau_k + S_T^{1/2},
	\end{split}
\end{equation*}
where we used that $(a+b)^{1/2} \leq a^{1/2} + b^{1/2}$ holds for $a, b \geq 0$ in the second inequality.
	By applying this bound to \eqref{eq7}, we get  
	\begin{align}\label{eq10} 
		\begin{split}
			&\sum_{k=1}^{T}(f_k(x_k)-f_k(x_k^*))
			\\
			&\leq Z_0 +D\sum_{k=1}^{T}\epsilon_k+\frac{1}{\lambda}\bigg( G_\omega-\frac{\sigma_\omega}{2}\bigg)\sum_{k=1}^{T}s_k^2 \\
			&\quad+\Big( \sum_{k=1}^T \tau_k + S_T^{1/2}\Big) \sum_{k=0}^{T}\bigg(e_k+\frac{G_\omega s_{k+1}}{\lambda}+\frac{G_\omega\epsilon_k}{\lambda}\bigg).
		\end{split}
	\end{align}	
	Here we recall the following notations
	\begin{equation*}
		\begin{split}
			&\Sigma_T = \sum_{t=2}^T s_t, \quad \overline{\Sigma}_T = \sum_{t=1}^T s_t^2
			\\
			&E_T = \sum_{t=1}^T \mathbb{E}\|e_t\|,\quad P_T = \sum_{t=1}^T \epsilon_t,\quad \overline{P}_T = \sum_{t=1}^T \epsilon_t^2.
		\end{split}
	\end{equation*}
	Then, one may see from \eqref{definitions} that the following estimates hold:
	\begin{align*}
		S_T = O(1+P_T + \overline{\Sigma}_T)\quad \textrm{and}\quad
		\sum_{k=1}^T \tau_k = O(E_T + \Sigma_T + P_T).
	\end{align*} Putting these estimates in \eqref{eq10}, we get
	\begin{align*} 
		\begin{split}
			&\sum_{k=1}^{T}(f_k(x_k)-f_k(x_k^*))
			\\
			& = O(1+P_T + \overline{\Sigma}_T)  +O\Big( ((E_T + \Sigma_T + P_T) +(1+P_T + \overline{\Sigma}_T)^{1/2}) (E_T + \Sigma_T +  P_T)\Big)
			\\
			& = O\Big(1 + P_T + \overline{\Sigma}_T + (E_T +\Sigma_T + P_T)^2\Big),
		\end{split}
	\end{align*}	
	where we used $ab \leq \frac{1}{2}(a^2 + b^2)$ in the last estimate. This completes the proof of \mbox{Theorem \ref{thm-1-2}.} \qed
	
\section{Numerical simulation}
This section provides numerical experiments of the online mirror descent for composite optimization. First we consider a multiple regression problem for time-varying system identification. Next we study the dynamic foreground-background separation problem for video frames. 
\begin{example}
	Consider the following identification problem:
	\begin{align*}
		y_t=a_{1,t}x_{1,t}+a_{2,t}x_{2,t}+...+a_{n,t}x_{n,t}+w_t
	\end{align*}
	where $x_{i,t}\in\mathbb{R}^d$ are observable inputs, $w_t\in\mathbb{R}^d$ is an unobservable random error, $a_{i,t}\in\mathbb{R}$ are sparse coefficients to be estimated, and $y_t\in\mathbb{R}^d$ is the corresponding response variable at time $t$. We generate a problem that the observable input $x_{i,t}\in\mathbb{R}^{2}$ $(i=1,...,30)$ is sampled randomly and the corresponding output $y_{t}\in\mathbb{R}^{2}$ is obtained from a Gauss-Markov model \cite{LL} where $w_t$ is given as $w_t\sim \mathcal{N}(0,0.01\mathbb{I}_{d\times d})$ and the coefficient $a_{i,t}$ is given by 
	\begin{align*}
		a_{i,t}=\begin{cases*}
			\alpha a_{i,t-1}+v_{i,t} \qquad & $i\in\{1,2\}$ \\
			0 &\text{otherwise}
		\end{cases*}
	\end{align*}
	Here $a_{i,0}\sim\mathcal{N}(0,1),\ \alpha=0.999$ and $v_{i,t}\sim\mathcal{N}(0,1-\alpha^2)$ for $i=1,2$. The aim is to estimate unknowns $a_{i,t}$ from the knowledge of $x_{i,t}$ and $y_{i,t}$, by solving the following time-varying optimization problem:
	\begin{align*}
		\min_{\bold{a}\in\mathbb{R}^{30}} \norm{y_t-X_t\bold{a}}^2_2+\eta \norm{\bold{a}}_1, \qquad t\in\mathbb{N},
	\end{align*}
	where $X_t=(x_{1,t},...,x_{30,t})$, $\bold{a}=(a_{1},...,a_{30})$ and the regularizing parameter $\eta=0.05$. For this, we use the online proximal mirror descent method \eqref{algorithm} with the step-size $\lambda=0.01$. Furthermore, to investigate the effects of inexact computations of gradients and proximal parts, we introduce artificial errors in our tests. Specifically, at each step $t$, our algorithm is implemented as follows:
	\begin{align}\label{ex1al}
		\begin{split}
		\bold{a}_{t+1/2}&=\bold{a}_t-\lambda (2(X_t\bold{a}_t-y_{t})\cdot X+e_t)\\
		\bold{a}_{t+1}&=\underset{\bold{a}\in\mathbb{R}^n}{\arg\min}\left\{\eta\|\bold{a}\|_1+\frac{1}{2\lambda}\|\bold{a}-\bold{a}_{t+1/2}\|^2_2\right\}+\epsilon_t.
		\end{split}
	\end{align}
	
	For the comparison, we first test the algorithm when the errors of gradient $e_t$ and mirror descent $\epsilon_t$ are given as identically zeros. In Figure 1, we plot the true values of $\{a_{1,t},a_{2,t},a_{6,t},a_{30,t}\}$ and the values predicted by the above algorithm. Next, we perform the same test with errors $e_t$ and $\ep_t$ in the gradient and proximal operator generated by the normal distribution $\mathcal{N}(0,0.05^2)$. The result is exhibited in Figure 2. Additionally, the dynamic regrets from both tests are presented in Figure 3.
\begin{figure}[hbt!]
	\centering
	\includegraphics[width=13cm]{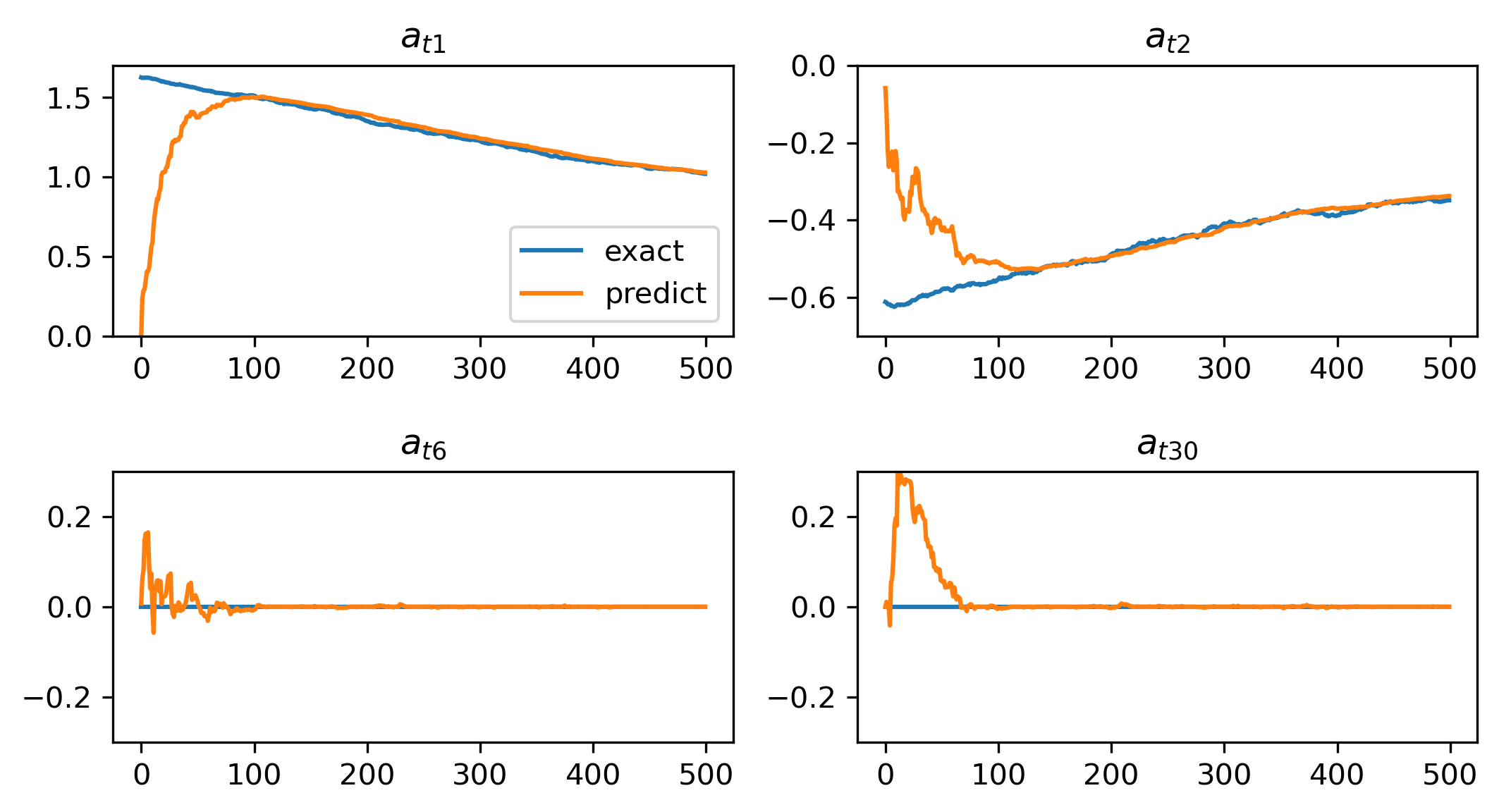}
	\caption{The algorithm \eqref{ex1al} is implemented with exact computation, i.e., $e_k,\epsilon_k\equiv0$. Here, the x-axis represents the number of iterations, and the y-axis represents the values of predicted and exact coefficient $a_{i,t}$. This figure demonstrates the algorithm's performance in finding coefficients for a time-varying linear model. Remarkably, predictions by the algorithm are very close to the exact coefficients after approximately 100 iterations.}
\end{figure}
	\begin{figure}[hbt!]
		\centering
		\includegraphics[width=13cm]{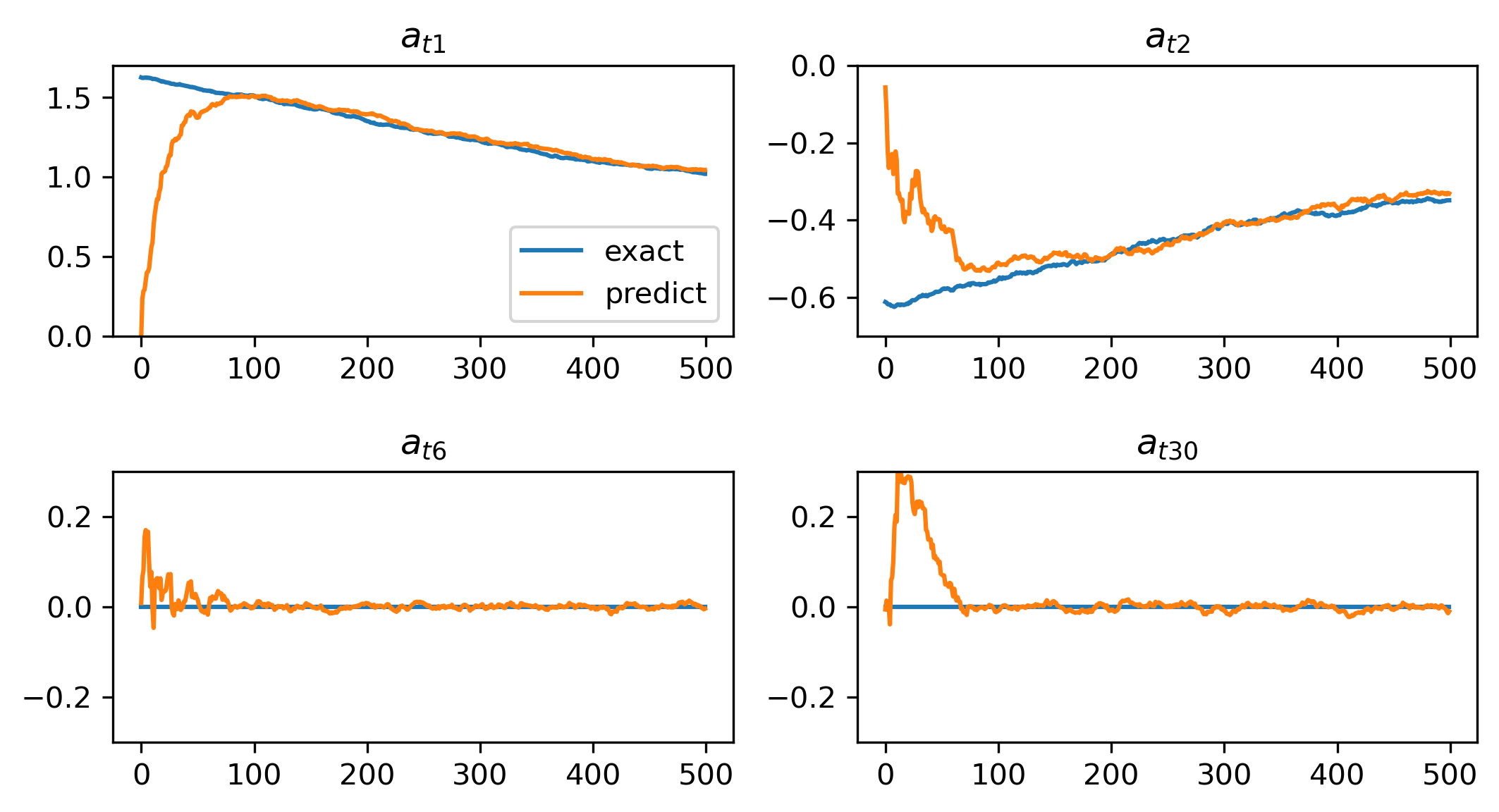}
		\caption{The algorithm \eqref{ex1al} is implemented with inexact computation, i.e., $e_k,\epsilon_k \sim \mathcal{N}(0,0.05^2)$.  Although the performance may not reach the level in the previous one, this figure shows that the algorithm still manages to achieve reasonable performance in the presence of the noise.}
	\end{figure}
\begin{figure}[hbt!]
	\centering
	\includegraphics{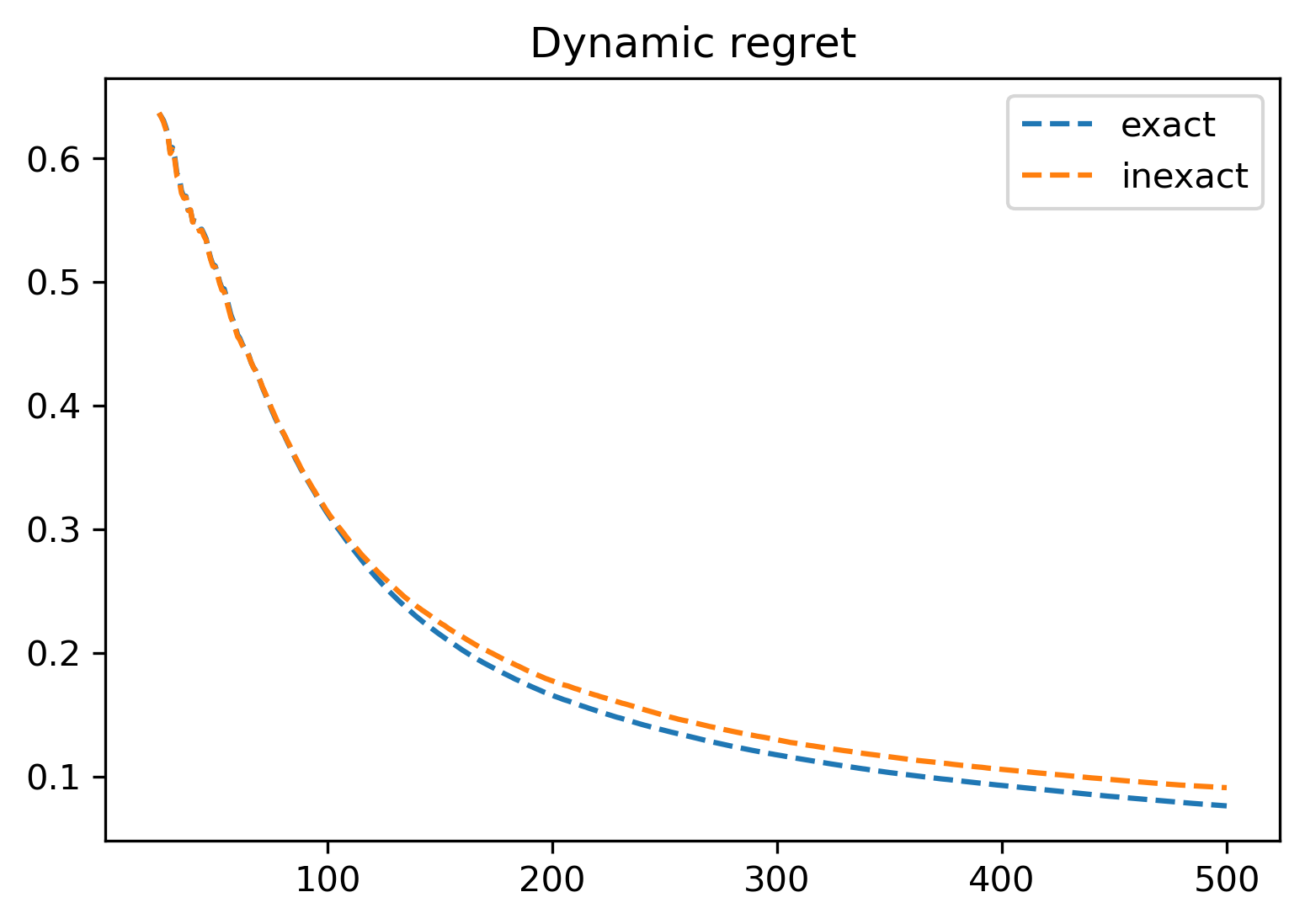}
	\caption{This figure presents a comparison of the dynamic regrets obtained from the two conducted experiments, one without errors and the other with uniformly distributed errors in the gradient and the proximal mirror operator. As the number of steps increases, the difference in dynamic regrets between the two experiments becomes more apparent. However, despite the presence of noise, the dynamic regret obtained by inexact computations still decays at a similar rate.}
\end{figure}
\end{example}
 
\begin{example}
Here we consider the foreground-background seperation problem for video frames, which was considered in \cite{DSR,MBCR}. 
At each moment, we gather a set of $L$ video frames $\{m_k \in \mathbb{R}^r\}_{k=1}^{L}$ into a matrix $M_k \in \mathbb{R}^{L\times r}$. We aim to split $M_k$ into a low-rank matrix $L_k$ representing the background and a sparse matrix $S_k$ representing the foreground as below:
\begin{equation*}
M_k = L_k + S_k.
\end{equation*}
Following the work \cite{DSR,MBCR}, we solve this problem by considering the following online composite optimization problem with streaming data $M_k$:
\begin{align*}
	\underset{L,S\in\mathbb{R}^{r\times L}}{\arg\min} \left\{ \|M_k -L-S\|_F^2 + \mu_{L} \|L\|_F^2 + \mu_s \|S\|_F^2+\lambda_L \|L\|_{*} + \lambda_s \|\textrm{vec}(S)\|_{1}\right\}
\end{align*}
where $\|\cdot\|_F$ denotes the Frobenious norm and $\|\cdot\|_{*}$ is the nuclear norm. In this example, we choose the parameters as $\mu_L =0.005$, $\mu_S =2.0$, $\lambda_L = 10^5$, and $\lambda_S = 0.034$. To solve this problem, we utilize the proximal online gradient descent given as
\begin{equation*}
\begin{split}
Z_{k+1}& = L_k - \alpha_{L} \Big( 2(L_k +S_k -M_k) + 2 \mu_L L_k\Big)
\\
L_{k+1}& = \mathcal{D}_{\alpha_L \lambda_L} (Z_{k+1})
\\
Y_{k+1}& = S_k - \alpha_S \Big(2 (L_k +S_k -M_k) + 2\mu_S S_k\Big)
\\
S_{k+1}& = \mathcal{S}_{\alpha_S \lambda_S} (Y_{k+1}).
\end{split}
\end{equation*}
where we used the step-sizes as $\alpha_L = 0.2$ and $\alpha_S = 0.2$. The proximal operators $\mathcal{D}$ and $\mathcal{S}$ are defined as follow. For a matrix $Z$, admitting the singular value decomposition $Z = U \Sigma V^T$, we set $\mathcal{D}_{\lambda}(Z) = U\mathcal{D}_{\lambda}(\Sigma) V^T$ where $\mathcal{D}_{\lambda}(\Sigma)$ is a diagonal matrix with entries given as $[\mathcal{D}_{\lambda}(\Sigma)]_{ii} = \max \{[\Sigma]_{ii}-\lambda, 0\}$. Next, the proximal operator $\mathcal{S}_{\lambda}(Y)$ for  matrix $Y$ has entries given as
\begin{equation*}
[\mathcal{S}_{\lambda}(Y)]_{ij} = \textrm{sign}([Y]_{ij})(|[Y]_{ij}| - \lambda)_{+}.
\end{equation*}
In our experiment, we utilized video data sourced from \cite{video}. To compute the proximal operator $\mathcal{D}$ effectively, one needs to use an approximate pair $(U,\Sigma,V)$ of the singular value decomposition such that $Z\simeq U\Sigma V$. This approximation brings some errors in the proximal operator. In other words, the computations are performed inexactly. The results of applying the subtraction algorithm to the video are shown in Figures \ref{fig1} and \ref{fig2}. Figure \ref{fig1} effectively demonstrates the algorithm's performance to accurately separate the background and foreground elements of the images. The dynamic regret is given in Figure \ref{fig2}, which illustrates that the dynamic regret decays well over time. These results emphasize the algorithm's adaptability and robustness when tackling real-world challenges involving inexact computations.
\begin{figure}[h]
	\centering
	\includegraphics[height=2.5cm, width=4.5cm]{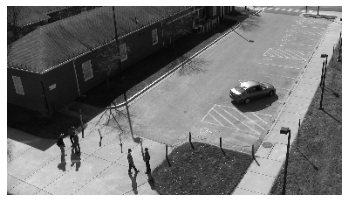}
	\includegraphics[height=2.5cm, width=4.5cm]{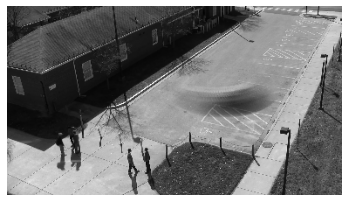}
	\includegraphics[height=2.5cm, width=4.5cm]{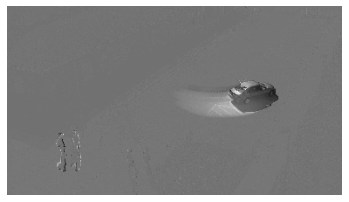}
	\includegraphics[height=2.5cm, width=4.5cm]{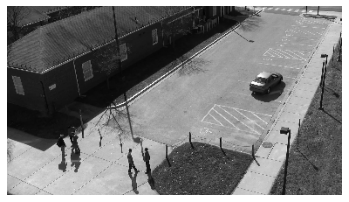}
	\includegraphics[height=2.5cm, width=4.5cm]{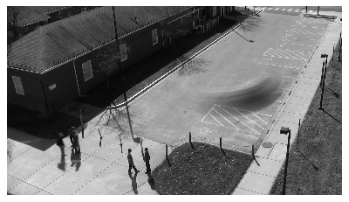}
	\includegraphics[height=2.5cm, width=4.5cm]{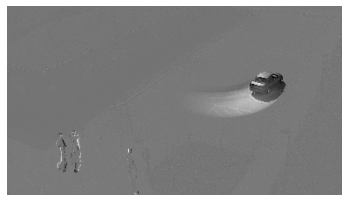}
	\caption{Background subtraction results at two different time points. Each row represents a distinct time point. (Left) Original images captured at the respective time points, (Middle) Background images obtained by the algorithm, (Right) Foreground images obtained by the algorithm.}
	\label{fig1}
	\vspace{-0.3cm}
\end{figure}

\begin{figure}[h]
	\centering
	\includegraphics[height=5cm, width=6cm]{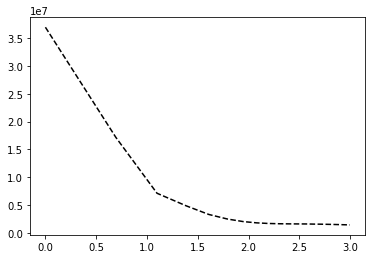}
	\caption{This figure presents the dynamic regret for the background subtraction experiment.}
	\label{fig2}
\end{figure}
\end{example}

\section*{Acknowledgment}
The work of Woocheol Choi was supported by the National Research Foundation of Korea NRF- 2016R1A5A1008055 and Grant NRF-2021R1F1A1059671. The work of Seok-Bae Yun was supported by Samsung Science and Technology Foundation under Project Number SSTF-BA1801-02.


\begin{thebibliography}{99}

\bibitem{ASD} Ajalloeian A, Simonetto A, and Dall’Anese E. Inexact online proximal-gradient method for time-varying convex optimization. In2020 American Control Conference (ACC); 2020 Jul 1; p. 2850-2857. IEEE.

\bibitem{BD} Bastianello N, and Dall’Anese E. Distributed and inexact proximal gradient method for online convex optimization. In2021 European Control Conference (ECC); 2021 Jun 29; p. 2432-2437. IEEE.

\bibitem{BT} Beck A, and Teboulle M. Mirror descent and nonlinear projected subgradient methods for convex optimization. Operations Research Letters. 2003 May 1;31(3):167-75.

\bibitem{BMN} Ben-Tal A, Margalit T, and Nemirovski A. The ordered subsets mirror descent optimization method with applications to tomography. SIAM Journal on Optimization. 2001;12(1):79-108.

\bibitem{DSBM}
Dall'Anese E, Simonetto A, Becker S, and Madden L. Optimization and learning with information streams: Time-varying algorithms and applications. IEEE Signal Processing Magazine. 2020 May 4;37(3):71-83.

\bibitem{DSH}
Derenick J, Spletzer J, and Hsieh A. An optimal approach to collaborative target tracking with performance guarantees. Journal of Intelligent and Robotic Systems. 2009 Sep;56:47-67.

\bibitem{DSR}
Dixit R, Bedi AS, Tripathi R, and Rajawat K. Online learning with inexact proximal online gradient descent algorithms. IEEE Transactions on Signal Processing. 2019 Jan 1;67(5): p. 1338-1352.

\bibitem{DAJ} Duchi JC, Agarwal A, Johansson M, and Jordan MI. Ergodic mirror descent. SIAM Journal on Optimization. 2012;22(4):1549-78.

\bibitem{DSST} Duchi JC, Shalev-Shwartz S, Singer Y, and Tewari A. Composite objective mirror descent. InCOLT; 2010 Jun 27; p. 14-26.

\bibitem{HW} Hall EC, and Willett RM. Online convex optimization in dynamic environments. IEEE Journal of Selected Topics in Signal Processing. 2015 Feb 18;9(4): p. 647-662.

\bibitem{KMD} Kim S, Madden L, and Dall’Anese E. Convergence of the Inexact Online Gradient and Proximal-Gradient Under the Polyak-Łojasiewicz Condition. arXiv preprint arXiv:2108.03285. 2021.

\bibitem{KBB} Krichene W, Bayen A, and Bartlett PL. Accelerated mirror descent in continuous and discrete time. Advances in neural information processing systems. 2015;28.

\bibitem{LL}Li J, and Li X. Online sparse identification for regression models. Systems \& Control Letters. 2020 Jul 1;141:104710.

\bibitem{LY}
Lei Y, and Zhou DX. Convergence of online mirror descent. Applied and Computational Harmonic Analysis. 2020;48(1):343-73.

\bibitem{NL} Nedic A, and Lee S. On stochastic subgradient mirror-descent algorithm with weighted averaging. SIAM Journal on Optimization. 2014;24(1):84-107.

\bibitem{MY} Nemirovskij, A. S., and Yudin, D. B. Problem complexity and method efficiency in optimization. 1983.

\bibitem{NB} Nokleby M, and Bajwa WU. Stochastic optimization from distributed streaming data in rate-limited networks. IEEE transactions on signal and information processing over networks. 2018 Aug 19;5(1):152-67.

\bibitem{MBCR} Moore BE, Gao C, and Nadakuditi RR. Panoramic robust pca for foreground–background separation on noisy, free-motion camera video. IEEE Transactions on Computational Imaging. 2019 Jan 6;5(2):195-211.

\bibitem{video}
Oh S, et al. A large-scale benchmark dataset for event recognition in surveillance video. InCVPR 2011; 2011 Jun 20; p. 3153-3160. IEEE.

\bibitem{PMR}
Paternain S, Morari M, and Ribeiro A. A prediction-correction method for model predictive control. In2018 Annual American Control Conference (ACC); 2018 Jun 27;p. 4189-4194. IEEE.

\bibitem{R} Rabbat M. Multi-agent mirror descent for decentralized stochastic optimization. In2015 IEEE 6th International Workshop on Computational Advances in Multi-Sensor Adaptive Processing (CAMSAP); 2015 Dec 13; p. 517-520. IEEE.

\bibitem{RB} Raginsky M, and Bouvrie J. Continuous-time stochastic mirror descent on a network: Variance reduction, consensus, convergence. In2012 IEEE 51st IEEE Conference on Decision and Control (CDC); 2012 Dec 10; p. 6793-6800. IEEE.

\bibitem{SRB}Schmidt M, Roux N, and Bach F. Convergence rates of inexact proximal-gradient methods for convex optimization. Advances in neural information processing systems. 2011;24.

\bibitem{SDPLG} Simonetto A, Dall'Anese E, Paternain S, Leus G, and Giannakis GB. Time-varying convex optimization: Time-structured algorithms and applications. Proceedings of the IEEE. 2020 Jul 3;108(11):2032-48.

\bibitem{lasso}
Tibshirani, R. Regression shrinkage and selection via the lasso. Journal of the Royal Statistical Society: Series B (Methodological). 1996;58(1):267-288.

\bibitem{YHHX} Yuan D, Hong Y, Ho DW, and Xu S. Distributed mirror descent for online composite optimization. IEEE Transactions on Automatic Control. 2020 Apr 17;66(2):714-29.

\bibitem{Ziegel}
Ziegel, Eric R. The elements of statistical learning. 2003; p. 267-268.

\bibitem{ZDH}
Zhang Y, Dall’Anese E, and Hong M. Online proximal-ADMM for time-varying constrained convex optimization. IEEE Transactions on Signal and Information Processing over Networks. 2021 Feb 3;7:144-55.

\bibitem{ZHHT}
Zou, H., and Hastie, T. Regularization and variable selection via the elastic net. Journal of the royal statistical society: series B (statistical methodology). 2005;67(2):301-320.


	


\end{thebibliography}
\end{document}